\documentclass[a4paper,11pt]{amsart}
\pdfoutput=1 
\usepackage{amssymb, latexsym, amsmath, verbatim, amsthm, amscd}
\usepackage[shortcuts]{extdash}
\usepackage{microtype}
\usepackage{mathtools}
\usepackage{mathscinet}
\usepackage{subcaption}
\usepackage{tikz,color}
\usetikzlibrary{decorations.markings,arrows.meta}
\usepackage{graphicx}
\usepackage{hyperref}

\title[Short, highly imprimitive words]{Short, highly imprimitive words yield hyperbolic
  one-relator groups}

\author[Christopher H.\ Cashen]{Christopher H.\ Cashen\orcidA{}} 
\address{
Faculty of Mathematics,
University of Vienna,
1090 Vienna, Austria}
\email{\href{mailto:christopher.cashen@univie.ac.at}{christopher.cashen@univie.ac.at}}
\urladdr{\href{http://www.mat.univie.ac.at/~cashen}{\nolinkurl{http://www.mat.univie.ac.at/~cashen}}}
\urladdr{\href{https://orcid.org/0000-0002-6340-469X}{\nolinkurl{https://orcid.org/0000-0002-6340-469X}}}
\author[Charlotte Hoffmann]{Charlotte Hoffmann\orcidB{}}
\address{
Faculty of Mathematics,
University of Vienna,
1090 Vienna, Austria}
\email{\href{mailto:chahoffmann94@gmail.com}{chahoffmann94@gmail.com}}
\urladdr{\href{https://orcid.org/0000-0003-2027-5549}{\nolinkurl{https://orcid.org/0000-0003-2027-5549}}}
\thanks{The first author is supported by the Austrian Science Fund (FWF): P~30487\=/N35.} 
\date{\today}

\hypersetup{
    pdftitle={Short, highly imprimitive words yield hyperbolic
  one-relator groups},    
    pdfauthor={Christopher H. Cashen and Charlotte Hoffmann},     
    pdfkeywords={one-relator group, hyperbolic group, imprimitivity rank}, 
    colorlinks=true,       
    linkcolor=black,          
    citecolor=black,        
    filecolor=black,      
    urlcolor=black           
}
\theoremstyle{plain}
\newtheorem{theorem}{Theorem}[section]
\newtheorem{lemma}{Lemma}[section]
\newtheorem{proposition}{Proposition}[section]
\newtheorem{corollary}{Corollary}[section]

\theoremstyle{remark}

\theoremstyle{definition}

\def\makeautorefname#1#2{\expandafter\def\csname#1autorefname\endcsname{#2}}
\let\fullref\autoref

\makeautorefname{theorem}{Theorem} 
\makeautorefname{lemma}{Lemma} 
\makeautorefname{proposition}{Proposition} 
\makeautorefname{corollary}{Corollary} 
\makeautorefname{definition}{Definition}
\makeautorefname{example}{Example}
\makeautorefname{section}{Section}
\makeautorefname{subsection}{Section}
\makeautorefname{subsubsection}{Section}
\makeautorefname{claim}{Claim}

\makeatletter 
\let\c@lemma=\c@theorem 
\makeatother
\makeatletter 
\let\c@proposition=\c@theorem 
\makeatother
\makeatletter 
\let\c@corollary=\c@theorem 
\makeatother
\makeatletter 
\let\c@definition=\c@theorem 
\makeatother
\makeatletter 
\let\c@example=\c@theorem 
\makeatother
\makeatletter
\@addtoreset{claim}{proposition}
\makeatother
\makeatletter
\@addtoreset{claim}{lemma}
\makeatother
\mathtoolsset{centercolon} 

\makeatletter
\newsavebox\myboxA
\newsavebox\myboxB
\newlength\mylenA

\newcommand*\xoverline[2][0.75]{%
    \sbox{\myboxA}{$\m@th#2$}%
    \setbox\myboxB\null
    \ht\myboxB=\ht\myboxA%
    \dp\myboxB=\dp\myboxA%
    \wd\myboxB=#1\wd\myboxA
    \sbox\myboxB{$\m@th\overline{\copy\myboxB}$}
    \setlength\mylenA{\the\wd\myboxA}
    \addtolength\mylenA{-\the\wd\myboxB}%
    \ifdim\wd\myboxB<\wd\myboxA%
       \rlap{\hskip 0.5\mylenA\usebox\myboxB}{\usebox\myboxA}%
    \else
        \hskip -0.5\mylenA\rlap{\usebox\myboxA}{\hskip 0.5\mylenA\usebox\myboxB}%
    \fi}
\makeatother

\newcommand{\F}{\mathbb{F}}
\DeclareMathOperator{\Aut}{Aut}

\newcommand{\llangle}{\langle\!\langle}
\newcommand{\rrangle}{\rangle\!\rangle}

\newcommand{\bp}{o}

\definecolor{lime}{HTML}{A6CE39}
\DeclareRobustCommand{\orcidicon}{%
	\begin{tikzpicture}
          \draw[lime, fill=lime] (0,0)
	circle [radius=0.16] 
	node[white] {{\fontfamily{qag}\selectfont \tiny ID}};
	\draw[white, fill=white] (-0.0625,0.095) 
	circle [radius=0.007];
	\end{tikzpicture}
	\hspace{-2mm}
}

\foreach \x in {A, ..., Z}{%
	\expandafter\xdef\csname orcid\x\endcsname{\noexpand\href{https://orcid.org/\csname orcidauthor\x\endcsname}{\noexpand\orcidicon}}
}

\makeatletter
\@namedef{subjclassname@2020}{%
  \textup{2020} Mathematics Subject Classification}
\makeatother

\begin{document}
\begin{abstract}
We give experimental support for a conjecture of Louder and Wilton
saying that words of imprimitivity rank greater than two yield
hyperbolic one-relator groups.

\end{abstract}
\maketitle
\section{Introduction}
An element in a free group is \emph{primitive} if it is an element of
some basis, or free generating set.
Failure of primitivity can be quantified:
define the \emph{imprimitivity rank} of an element to be the minimal
rank of a subgroup containing it as an imprimitive element, if such a
subgroup exists, or infinite otherwise.
An element has imprimitivity rank 0 if and only if it is trivial, 1 if and
only if it is a proper power, and $\infty$ if and
only if it is a primitive element. 
In these cases the quotient of the free group by the subgroup
normally generated by the element is a hyperbolic
group, either a free group, in the first and third cases, or a
one-relator group with torsion, which is hyperbolic by
the B.\ B.\ Newman spelling theorem \cite{MR0222152}, in the second case.
Nonelementary, torsion-free two-generator one-relator groups have
relators of imprimitivity rank 2.
There are many nonhyperbolic
groups of this form, such as $\mathbb{Z}^2=\langle a,b\mid
aba^{-1}b^{-1}\rangle$, the Baumslag-Solitar
groups $BS(m,n)=\langle a,b\mid
ab^ma^{-1}b^{-n}\rangle$, and the groups considered by Gardam and
Woodhouse \cite{MR3876736}.
Louder and Wilton \cite[Theorem~1.4]{LouWil18a} show that
two-generated subgroups of a higher imprimitivity rank one-relator
group are free.
Thus, they are of Type $F$ and have no
Baumslag-Solitar subgroups.
It is a long-standing open question whether such groups must be hyperbolic. 
Louder and Wilton conjecture \cite[Conjecture~1.6]{LouWil18a} a positive answer for
one-relator groups.

We offer experimental support for their conjecture.
Fix a basis for a free group, so that a group element can be uniquely represented as a freely reduced word, a product of basis elements and their
inverses, of a well-defined length.
 \begin{theorem}\label{rank3}
    Let $w$ be a word in $\F_r$ of length $L$ and imprimivity
    rank not equal to 2. Then $\F_r/\llangle w\rrangle$ is hyperbolic
    if $r\leq 4$ and $L\leq 17$.
  \end{theorem}
 
  These results are achieved computationally, by a combination of efficient enumeration
  of representatives and brute force\footnote{We ran 12 x 4 core
    Intel Core i5-4670S @ 3.10GHz for two months.}.
  Details are in \fullref{sec:experiment}.

 We also observe that a well-known result about hyperbolicity of
 one-relator groups is consistent with the conjecture. In these
 results $|w|_a$ denotes the total number of occurrences of $a$ and
 $a^{-1}$ in $w$, where $w$ is freely reduced and $a$ is an element of the chosen basis.
 \begin{proposition}\label{prop:ISirank}
    The nonhyperbolicity criteria of Ivanov and Schupp \cite[Theorems~3~\&~4]{MR1401522}
    imply imprimitivity rank 2.
  \end{proposition}
  The proposition is proven in \fullref{sec:IS}.
As a consequence, we have:
  \begin{corollary}\label{IS}
    Let $w$ be a cyclically reduced word in $\F_r$ of imprimitivity rank not equal to 2
    such that $0<|w|_a<4$ for some basis element $a$. Then
    $\F_r/\llangle w\rrangle$ is hyperbolic. 
  \end{corollary}
  \begin{corollary}\label{ISbound}
    Let $w$ be cyclically reduced word in $\F_r$ of length less than $4r$ and
    imprimitivity rank not equal to 2 such that every generator or its
    inverse occurs in $w$.
    Then $\F_r/\llangle w\rrangle$ is hyperbolic.
  \end{corollary}

Combining these results with our experimental results, we have:
  \begin{corollary} 
    Let $w$ be a word in $\F_r$ of length at most 17 and imprimitivity
    rank not equal to 2. Then $\F_r/\llangle w\rrangle$ is hyperbolic.
  \end{corollary}
  \begin{proof}
    $\F_r/\llangle w\rrangle$ is hyperbolic when the imprimitivity
    rank of $w$ is 0, 1, or $\infty$, so suppose it is finite and at least
    3.
    Up to replacing $w$ by an element in the same automorphic orbit,
    we may, without increasing the length of $w$, assume that it is
    cyclically reduced and that there is $s$ such that taking the
    first $s$ basis elements and the last $r-s$ basis elements gives a
    splitting  $\F_r=\F_s*\F_{r-s}$ where the $\F_s$
    factor is the smallest free factor containing $w$.
    Since $w$ is imprimitive in $\F_r$, it is imprimitive in $\F_s$,
    so $s$ is an upper bound on imprimitivity rank, which implies
    $s\geq 3$.
    Furthermore, since $\F_s$ is the smallest free factor containing
    $w$, all of the generators of $\F_s$ or their inverses occur in $w$.
    Since $\F_r/\llangle w\rrangle\cong (\F_s/\llangle w\rrangle) *
    \F_{r-s}$ is hyperbolic if and only if $\F_s/\llangle w\rrangle$
    is, 
    we conclude by applying \fullref{rank3} or
    \fullref{ISbound} to $\F_s/\llangle w\rrangle$, according to
    whether  $s<5$ or $s\geq 5$, respectively.
  \end{proof}

  \subsection*{Additional conjectures}
  In checking the hyperbolicity conjecture, we enumerated the
  $\Aut(\F_4)$ orbits of cyclic subgroups of $\F_4$ that have a
  representative that can be generated by a word of at most length 16.
  We also computed imprimitivity ranks for these words.
  Armed with this data, we can test other questions involving
  imprimitivity rank.
  We check two additional conjectures and find that they are consistent with the
  data up to length 16. The first of these concerns uniqueness of the
  subgroup in the definition of imprimitivity rank, see
  \fullref{uniqueWsubgroups}.
  The second concerns the relationship between imprimitivity rank and
  stable commutator length, see \fullref{scl}.

  \subsection*{Acknowledgements}
We thank Henry Wilton for his comments on an earlier draft, in
particular for the suggestion to check Heuer's conjecture.

\section{Preliminaries}
Fix a free group $\F_r$ with basis $X=(x_1,\dots,x_r)$.
Let $X^\pm :=\{x_1,\dots,x_r\}\cup\{x_1^{-1},\dots,x_r^{-1}\}$.
Write $f\sim g$ if $f$ and $g$ are conjugate.
The word length of $f$ with respect to $X$ is denoted $|f|$, and the
word length of the cyclic reduction of $f$ with respect to $X$, the
\emph{cylic length of $f$}, is
denoted $||f||$.

For our purposes, a finitely presented group is \emph{hyperbolic} if
there exists a linear function $\delta$ such that if $w$ is a freely
reduced word of length $n$ in the generators or their inverses that represents the
identity element of the group then it is possible to express $w$ as
the free reduction of a product of at most $\delta(n)$ conjugates of
relators or their inverses.
It turns out that while the precise function $\delta$ depends on the
choice of finite presentation, its linearity does not, so being
hyperbolic is a group property and not merely a property of a
presentation. 
More on hyperbolic groups can be found in any textbook on Geometric
Group Theory. 

Imprimitivity rank was introduced by Puder\footnote{Puder 
  uses the term `primitivity rank'. Louder and Wilton follow his
  terminology. We find it misleading. Compare, for
  instance, to the \emph{primitivity
    index} of \cite{MR3893306}, which is the smallest index of a
  subgroup for which the element 
  \emph{becomes primitive} upon lifting to that subgroup. } \cite{MR3265279}.

A \emph{Stallings graph} is a based, directed, connected, $X$--labelled
graph $(\Gamma,\bp)$ that is folded and core with respect to $\bp$.
The free group $\pi_1(\Gamma,\bp)$ is identified with a subgroup of
$\F_r$ via the labelling, and, in fact, Stallings graphs are in bijection with subgroups of $\F_r$.
See Kapovich and Myasnikov  \cite{MR1882114} for details.

The group $W_I$ of \emph{Whitehead automorphisms of the first kind}
are automorphic
extensions of maps defined on $X$ by  $x_i\mapsto x_{\sigma(i)}^{\epsilon_i}$ for $1\leq i\leq
r$, where $\sigma\in\mathrm{Sym}(r)$ and $\epsilon_i=\pm 1$.

The set $W_{II}$ of  \emph{Whitehead automorphisms of the second kind}
are automorphic extensions of maps defined on $X^\pm$ as follows.
Given an element $x\in X^{\pm}$ and a subset $Z\subset
X^\pm\setminus\{x,x^{-1}\}$ take the map that
 fixes $x$ and $x^{-1}$ and for $y\in X^\pm\setminus\{x,x^{-1}\}$ does:
\[y\mapsto
  \begin{cases}
    y &\text{ if }y,y^{-1}\notin Z\\
    xy&\text{ if }y\in Z\text{ and }y^{-1}\notin Z\\
    yx^{-1}&\text{ if }y\notin Z\text{ and }y^{-1}\in Z\\
    xyx^{-1}&\text{ if }y,y^{-1}\in Z
  \end{cases}\]
Together the Whitehead automorphisms generate $\Aut(\F_r)$.
Moreover, Whitehead \cite{Whi36} proves two stronger facts:
\begin{itemize}
\item Call a word $w\in\F_r$ \emph{Whitehead minimal} if there
  does not exist a Whitehead automorphism $\alpha$ such that
  $||\alpha(w)||<|w|$.
  An element has minimal length in its $\Aut(\F_r)$ orbit if
  and only if it is Whitehead minimal.

\item Define the \emph{Whitehead level--$L$ graph} to be the graph
  whose vertices are Whitehead minimal words of length $L$, where $w$
  and $v$ are connected by an edge if there exists a Whitehead
  automorphism $\alpha$ such that $v$ is the cyclic reduction of
  $\alpha(w)$. Then the partition of vertices by connected component
  in the Whitehead level--$L$ graph is the same as the partition by
  $\Aut(\F_r)$ orbits.
  \end{itemize}
  Combining these two facts gives Whitehead's Algorithm for determining
  if two words are in the same $\Aut(\F_r)$ orbit: they are if and
  only if their Whitehead minimal representatives have the same
  length, say $L$,
  and are contained in the same component of the Whitehead level--$L$
  graph. 
  In particular, a word represents a primitive element if and only if
  it  Whitehead reduces to a word of  length 1.

\section{The Ivanov-Schupp criteria}\label{sec:IS}
\begin{theorem}[{\cite[Theorem~3]{MR1401522}}]\label{ISthm3}
	Let $w$ be a freely and cyclically reduced word in $\F_r$ and suppose
        that for some basis element $a$, the total number of
        occurrences of $a$ and $a^{-1}$, $|w|_a$, satisfies
        $0<|w|_a<4$.
        The group $\F_r/\llangle w\rrangle$ is not hyperbolic if and only if one of the following holds up to cyclic permutation and taking inverses:
	\begin{enumerate}
		\item $|w|_a=2$, $w=auav$ and $uv^{-1}$ is a proper power in $\F_r$.
		\item $|w|_a=2$, $w=aua^{-1}v$ and either $u$ and $v$ are conjugate to powers of the same word in $\F_r$ or $u$ and $v$ are both proper powers in $\F_r$.
		\item $|w|_a=3$, $w=atauav$ and $ut^{-1}=z^{m}$, $vt^{-1}=z^{n}$ where $z$ is not a proper power, such that one of the following holds:
		\begin{enumerate}
			\item $\min(|m|,|n|)=0$ and $\max(|m|,|n|)>1$.\label{ISitem1}
			\item $\min(|m|,|n|)>0$ and $|m|=|n|\neq 1$. \label{ISitem2}
			\item $\min(|m|,|n|)>0$ and $m=-n$. \label{ISitem3}
			\item $\min(|m|,|n|)>0$ and $m=2n$ (or $n=2m$).\label{ISitem4}
		\end{enumerate}
		\item $|w|_a=3$, $w=ataua^{-1}v$ and $t^{-1}ut=z^{m}$, $v=z^{n}$ where $z$ is not a proper power and either $|m|=|n|$ or $m=-2n$ (or $n=-2m$). 
	\end{enumerate}
\end{theorem}

\begin{theorem}[{\cite[Theorem~4~(3)]{MR1401522}}]\label{ISthm4}
	Let $w=au_{1}au_{2}au_{3}au_{4}$ be a freely and cyclically reduced word
        in $\F_r$, such that $|w|_a=4$ and the subwords $u_{i}$ are
        pairwise different.
        Then the group $\F_r/\llangle w\rrangle$ is not hyperbolic if and only if for some $i\in\{1,\dots,4\}$ the following holds (with subscripts modulo $4$): $$u_{i}u_{i+1}^{-1}u_{i+2}u_{i+3}^{-1}=1.$$
\end{theorem}

We check that nonhyperbolicity in these theorems implies
imprimitivity rank 2:
\begin{proof}[Proof of \fullref{prop:ISirank}]
Suppose $w$ is of one of the forms in \fullref{ISthm3} and
\fullref{ISthm4}.
For each case we exhibit a connected, based, rank 2 core graph with edges labelled by
words in $\F_r$ in which (a conjugate of) $w$ labels a
imprimitive element of the fundamental group.
By subdividing edges we can arrange that edges are labelled by basis
elements.
The graphs are not necessarily folded, but from the hypothesis in
\fullref{ISthm3} and \fullref{ISthm4} that the only occurrences of
$a^{\pm 1}$ are the explicit ones, it follows that in all of our
examples folding will be a homotopy equivalence, so these graphs
really do represent rank 2 subgroups.

In each of the figures the larger dot marks the base vertex, the
triangular arrows mark a choice of edges in a maximal
subtree, and the edges with the single and double arrows mark edges
whose unique completion through the maximal subtree to a based loop
represent 
generators $\alpha$ and $\beta$, respectively, of the fundamental
group of the graph.

First, let $|w|_a=2$ and $w=auav$ such that $uv^{-1}=x^{n}$ with $n>1$.
Then $w\sim (va)^{2}x^{n}\cong \alpha^2\beta^n$ is
imprimitive in \fullref{IS12A}.
	
	Now assume that $|w|_a=2$, $w=aua^{-1}v$ and
        $u=s_{1}^{-1}x^{m}s_{1}$ and $v=s_{2}^{-1}x^{n}s_{2}$.
        Note that $\min(|m|,|n|)>0$, since otherwise $w$ fails to be
        either freely or cyclically reduced, so $w\cong \alpha\beta^m\alpha^{-1}\beta^n$
        is imprimitive in \fullref{IS12B}.
        If $u=u_{0}^{m}$ and $v=v_{0}^{n}$ with $\min(m,n)>1$ then $w\cong \alpha^m\beta^n$ is imprimitive 
        in \fullref{IS12C}.
	
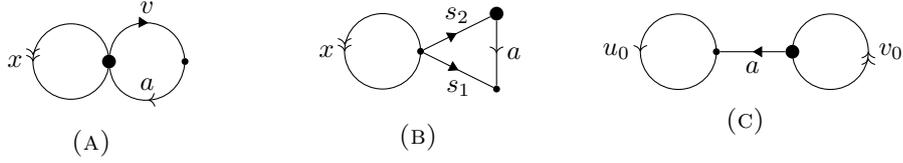
\begin{figure}
  \begin{subfigure}[c]{0.32\linewidth}
    \centering
     \begin{tikzpicture}[font=\small]
    \draw[decoration={markings, mark=at position 0.5 with {\arrow{>>}}},postaction={decorate}](.5,.5) circle (0.5);
   \node[shape=circle,fill=black,scale=.25] (bp) at (2,.5) {};
   \node[shape=circle,fill=black,scale=.5] (v) at (1,.5) {};
   \node[left] () at (0,.5) {$x$};
   \begin{scope}[decoration={markings,mark=at position 0.5 with {\arrow{Triangle}}}] 
      \draw[postaction={decorate}] (v) .. controls ++(83:.666) and ++(97:.666)
    .. node[above]{$v$} (bp);
   \end{scope}
   \begin{scope}[decoration={markings,mark=at position 0.5 with {\arrow{>}}}] 
      \draw[postaction={decorate}] (bp) .. controls ++(263:.666) and ++(277:.666)
    .. node[above]{$a$} (v);
   \end{scope}
 \end{tikzpicture}
   \caption{}\label{IS12A}
 \end{subfigure}
 \hfill
\begin{subfigure}[c]{0.32\linewidth}
  \centering
  \begin{tikzpicture}[font=\small]
     \draw[decoration={markings, mark=at position 0.5 with {\arrow{>>}}},postaction={decorate}](.5,.5) circle (0.5);
    \node[shape=circle,fill=black,scale=.25] (v1) at (2,0) {};
    \node[shape=circle,fill=black,scale=.5] (v2) at (2,1) {};
    \node[shape=circle,fill=black,scale=.25] (v0) at (1,.5) {};
    \node[left] () at (0,.5) {$x$};
    \begin{scope}[decoration={markings,mark=at position 0.5 with {\arrow{Triangle}}}] 
      \draw[postaction={decorate}] (v0) -- node[above]{$s_2$} (v2);
      \draw[postaction={decorate}] (v0) -- node[below]{$s_1$} (v1);
   \end{scope}
   \begin{scope}[decoration={markings,mark=at position 0.5 with {\arrow{>}}}] 
      \draw[postaction={decorate}] (v2) -- node[right]{$a$} (v1);
   \end{scope}
 \end{tikzpicture}
   \caption{}\label{IS12B}
 \end{subfigure}
 \hfill
\begin{subfigure}[c]{0.32\linewidth}
  \centering
\begin{tikzpicture}[font=\small]
     \draw[decoration={markings, mark=at position 0.5 with
       {\arrow{>}}},postaction={decorate}](.5,.5) circle (0.5);
     \draw[decoration={markings, mark=at position 0 with {\arrow{>>}}},postaction={decorate}](2.5,.5) circle (0.5);
    \node[shape=circle,fill=black,scale=.25] (v1) at (1,.5) {};
    \node[shape=circle,fill=black,scale=.5] (v0) at (2,.5) {};
    \node[left] () at (0,.5) {$u_0$};
    \node[right] () at (3,.5){$v_0$};
    \begin{scope}[decoration={markings,mark=at position 0.5 with {\arrow{Triangle}}}]
    \draw[postaction={decorate}] (v0) -- node[below]{$a$} (v1);
    \end{scope}
  \end{tikzpicture}
  \caption{}\label{IS12C}
\end{subfigure}
\caption{Graphs for $|w|_a=2$ in the proof of
 \fullref{prop:ISirank}.}\label{fig:IScase12}
\end{figure}

	Let now $|w|_a=3$, $w=atauav$ with $ut^{-1}=z^{m}$ and 
        $vt^{-1}=z^{n}$, where $z$ is not a proper power and $m,n$
        satisfy one of the conditions (\ref{ISitem1})-(\ref{ISitem4})
        in \fullref{ISthm3}.
      
        Suppose in case \eqref{ISitem1} we have $m=0$ and $n>1$, other
        variations of this case being similar.
        Then $w\cong \alpha^3\beta^n$ is imprimitive in \fullref{IS34A}.

        In case \eqref{ISitem3}, $w\cong
        \alpha^2\beta\alpha\beta^{-1}$ is imprimitive 
        in \fullref{IS34B}.

        In the subcase $m=n$ of case \eqref{ISitem2} that is not
        covered by case \eqref{ISitem3}, we may assume $m>1$ by
        replacing $z$ with $z^{-1}$, if necessary.
        Then $w\cong \alpha^2\beta^m\alpha\beta^m$ in \fullref{IS34A}.
        This word admits a Whitehead reduction $\alpha^{-1}\mapsto
        \beta\alpha^{-1}$, which sends the $w$--loop to
        $\alpha\beta^{-1}\alpha\beta^{2(m-1)}$.
        Since $m>1$, this word is Whitehead minimal, so the $w$--loop
        is imprimitive.

        In case \eqref{ISitem4} assume $n=2m$, the other case being
        similar.
        Then $w\cong \alpha^2\beta\alpha\beta^2$ is imprimitive
        in \fullref{IS34B}.
	
	Next, consider the case $|w|_a=3$, $w=ataua^{-1}v$ and
        $t^{-1}ut=z^{m}$, $v=z^{n}$ where $z$ is not a proper power.
        Again we can assume that $|m|,|n|>0$ since otherwise $|w|_a$
        would be less than $3$.
        Then $w=(at)^2z^{m}(at)^{-1}z^{n}$.
        Consider \fullref{IS34C}.
        If $|m|=|n|$ then $w\cong
        \alpha^2\beta\alpha^{-1}\beta^{\pm 1}$.
        If $n=-2m$ then $w\cong\alpha^2\beta\alpha^{-1}\beta^{-2}$.
        In all three cases $w$ is
        imprimitive.
        If $m=-2n$ then $w\cong\alpha^2\beta^2\alpha^{-1}\beta^{-1}$ is
        imprimitive in the graph obtained from \fullref{IS34C} by
        relabelling the $\beta$ edge with $z^{-n}$.

\begin{figure}
\begin{subfigure}[c]{0.31\linewidth}
 \centering
 \begin{tikzpicture}[font=\small]
    \draw[decoration={markings, mark=at position 0.5 with {\arrow{>>}}},postaction={decorate}](.5,.5) circle (0.5);
   \node[shape=circle,fill=black,scale=.25] (bp) at (2,.5) {};
   \node[shape=circle,fill=black,scale=.5] (v) at (1,.5) {};
   \node[left] () at (0,.5) {$z$};
   \begin{scope}[decoration={markings,mark=at position 0.5 with {\arrow{Triangle}}}] 
      \draw[postaction={decorate}] (v) .. controls ++(83:.666) and ++(97:.666)
    .. node[above]{$t$} (bp);
   \end{scope}
   \begin{scope}[decoration={markings,mark=at position 0.5 with {\arrow{>}}}] 
      \draw[postaction={decorate}] (bp) .. controls ++(263:.666) and ++(277:.666)
    .. node[below]{$a$} (v);
   \end{scope}
 \end{tikzpicture}
 \caption{}\label{IS34A}
\end{subfigure}
\begin{subfigure}[c]{0.31\linewidth}
 \centering
 \begin{tikzpicture}[font=\small]
    \draw[decoration={markings, mark=at position 0.5 with {\arrow{>>}}},postaction={decorate}](.5,.5) circle (0.5);
   \node[shape=circle,fill=black,scale=.25] (bp) at (2,.5) {};
   \node[shape=circle,fill=black,scale=.5] (v) at (1,.5) {};
   \node[left] () at (0,.5) {$z^m$};
   \begin{scope}[decoration={markings,mark=at position 0.5 with {\arrow{Triangle}}}] 
      \draw[postaction={decorate}] (v) .. controls ++(83:.666) and ++(97:.666)
    .. node[above]{$t$} (bp);
   \end{scope}
   \begin{scope}[decoration={markings,mark=at position 0.5 with {\arrow{>}}}] 
      \draw[postaction={decorate}] (bp) .. controls ++(263:.666) and ++(277:.666)
    .. node[below]{$a$} (v);
   \end{scope}
 \end{tikzpicture}
 \caption{}\label{IS34B}
\end{subfigure}
\begin{subfigure}[c]{0.31\linewidth}
  \centering
 \begin{tikzpicture}[font=\small]
    \draw[decoration={markings, mark=at position 0.5 with {\arrow{>>}}},postaction={decorate}](.5,.5) circle (0.5);
   \node[shape=circle,fill=black,scale=.5] (bp) at (1,.5) {};
   \node[shape=circle,fill=black,scale=.25] (v) at (2,.5) {};
   \node[left] () at (0,.5) {$z^m$};
   \begin{scope}[decoration={markings,mark=at position 0.5 with {\arrow{Triangle}}}] 
      \draw[postaction={decorate}] (v) .. controls ++(97:.666) and ++(83:.666)
    .. node[above]{$t$} (bp);
   \end{scope}
   \begin{scope}[decoration={markings,mark=at position 0.5 with {\arrow{>}}}] 
      \draw[postaction={decorate}] (bp) .. controls ++(277:.666) and ++(263:.666)
    .. node[below]{$a$} (v);
   \end{scope}
 \end{tikzpicture}
 \caption{}\label{IS34C}
\end{subfigure}
	\caption{Graphs for $|w|_a=3$ in the proof of \fullref{prop:ISirank}.}\label{fig:IScase34}
	\end{figure}
	
Finally, let $w=au_{1}au_{2}au_{3}au_{4}$ be as in \fullref{ISthm4}. 
We may assume that
        $u_{1}u_{2}^{-1}u_{3}u_{4}^{-1}=1$, so
        $w=au_{1}au_{2}au_{3}au_{1}u_{2}^{-1}u_{3}\cong
        \alpha\beta\alpha\beta^{-1}$ is imprimitive 
        in \fullref{fig:IScase5}.
 \begin{figure}
   \centering
   \begin{tikzpicture}[font=\small]
     \node[shape=circle,fill=black,scale=.25] (v0) at (0,0) {};
     \node[shape=circle,fill=black,scale=.25] (v1) at (1,0) {};
     \node[shape=circle,fill=black,scale=.25] (v2) at (1,1) {};
     \node[shape=circle,fill=black,scale=.5] (v3) at (0,1) {};
     \begin{scope}[decoration={markings,mark=at position 0.5 with
        {\arrow{Triangle}}}] 
        \draw[postaction={decorate}] (v0) -- node[below] {$u_1$} (v1);
        \draw[postaction={decorate}] (v2) -- node[above] {$u_3$}  (v3);
        \draw[postaction={decorate}] (v3) -- node[left] {$a$}  (v0);
      \end{scope}
      \begin{scope}[decoration={markings,mark=at position 0.5 with {\arrow{>}}}] 
        \draw[postaction={decorate}] (v1) -- node[right] {$a$}  (v2);
      \end{scope}
      \begin{scope}[decoration={markings,mark=at position 0.5 with {\arrow{>>}}}] 
        \draw[postaction={decorate}] (v2) .. controls ++(-7:.666) and ++(7:.666)
    .. node[right]{$u_2$} (v1);
      \end{scope}
   \end{tikzpicture}
		\caption{A  graph for $|w|_a=4$ in the proof of \fullref{prop:ISirank}.}\label{fig:IScase5}
	\end{figure}
\end{proof}

\section{The experiments}\label{sec:experiment}
\nocite{ipython}

To prove \fullref{rank3}, the idea is to enumerate
words of each length in the given free group, compute their
imprimitivity ranks, and for those of imprimitivity rank not equal to
two, test to see if
the resulting one-relator presentation is a hyperbolic group. 

\subsection{Enumerating words/groups}
For $w\in\F_r$ an automorphism $\alpha\in\Aut(\F_r)$ induces an
isomorphism between $\F_r/\llangle
w\rrangle$ and $\F_r/\llangle \alpha(w)^{\pm 1}\rrangle$.
Call these the `obvious' isomorphisms between one-relator groups.
To enumerate isomorphism types of one-relator groups it suffices to
enumerate one generator of one
representative of each automorphic orbit of cyclic subgroup.
There is a canonical choice of such an element: we choose the one that
is shortlex minimal with respect to the integer lexicographic order;
that is, if $(a_1,\dots,a_r)$ is our fixed ordered basis for $\F_r$,
we declare $a_r^{-1}<a_{r-1}^{-1}<\dots < a_1^{-1}<a_1<\dots <a_r$ and
extend to a shortlex ordering on reduced words.
There are examples of McCool and Pietrowski
\cite{MR0399269} that show that not all isomorphisms between
one-relator groups are obvious, so our enumeration has some
redundancies at the level of isomorphism type of one-relator groups.
However, work of Kapovich and Schupp \cite{MR2107437} and Kapovich,
Schupp, and Shpilrain \cite{MR2221020}, says that there is a generic
set of one-relator groups for which the only isomorphisms are the
obvious ones, so the redundancies are rare, in a specific quantifiable
sense. 

A naive algorithm for enumerating representatives of length $L$ is to simply
construct the Whitehead level--$L$ graph.
Additionally, since we are interested in cyclic subgroups and not just
elements, we connect every vertex $v$ to the vertex $v^{-1}$.
Then choose the shortlex minimal word in each component.

We speed this algorithm up as follows.
Permutation of generators and inversion of
generators and conjugation by a generator are in $\Aut(\F_r)$.
Define the \emph{PCI class} of a word to be those words that can be
reached from it by a finite chain of \emph{P}ermutation of generators,
\emph{C}yclic permutation, or \emph{I}nversion of
generators.
Similarly, the \emph{PCI$^\pm$ class} is those words that can be reached
by the above operations plus group inversion.  
Define a word to be \emph{SLPCI$^{(\pm)}$ minimal} if it is
\emph{S}hort\emph{L}ex minimal in its PCI$^{(\pm)}$ class.
Notice that if we start with a cyclically reduced word then none of
the above operations change the length of the word.

\begin{lemma}\label{slpci}
 $\Aut(\F_r)$ equivalence classes of cyclic subgroup such that the
 minimal generator length of a representative has length $L$ are in
 bijection with connected components of the length--$L$  \emph{SLPCI$^\pm$ graph}:
 the graph whose vertices are
 freely and cyclically reduced words of $\F_r$ of length $L$ that are both
 Whitehead and SLPCI$^\pm$
 minimal, and where two vertices $u$
 and $v$ are connected by an edge if there exists an element
 $\alpha\in W_{II}$ such that $v$ is the 
 SLPCI$^\pm$ minimal representative of $\alpha(u)$.
\end{lemma}
Khan \cite{MR2077762} used a similar construction, without inversion, to study the
complexity of Whitehead's Algorithm in the special case
$r=2$.
\begin{proof}
  Whitehead's result shows that the partition by components of the
  Whitehead level graph is the same as the partition by
  $\Aut(\F_r)$--orbits.
  It is clear from the definitions that two words in the same
  component of the length--$L$ SLPCI$^\pm$ graph are in the same component of the
  Whitehead level--$L$ graph.
  We show the opposite.
  The essential observation is that $W_I$ acts by conjugation on the set $W_{II}$.

  Elements that differ by an element of $W_I$ are in the same PCI
  class, so suppose $\alpha\in W_{II}$ and $\alpha(u_1)=u_2$ where
$u_i=a_iv_ia_i^{-1}$ with $v_i$ cyclically reduced, and suppose
$\sigma_i(v_i^{\epsilon_i})=w_i$ is the SLPCI$^\pm$ minimal
representative of $v_i$, where $\sigma_i\in W_I$ and $\epsilon_i\in\pm
1$.
Let $\alpha':=\sigma_1\circ\alpha\circ\sigma_1^{-1}$.
Since $\alpha'$ is a $W_I$ conjugate of an element of $W_{II}$, 
$\alpha'\in W_{II}$.
Thus there is an element of $W_{II}$ that takes $w_1$ to $\alpha'(w_1)\sim\sigma_1(u_2^{\epsilon_1})$, which
is in the same
PCI$^\pm$ class as $u_2$, so
$w_2$ is the SLPCI$^{\pm}$ minimal representative of $\alpha'(w_1)$.
\end{proof}

The lemma says we can run the naive algorithm but instead of enumerating
all words of a fixed length, it's enough to enumerate SLPCI$^\pm$
minimal ones.
This is a benefit because  SLPCI$^\pm$  minimality is falsifiable by a
subword: if $w$ is a word that contains a prefix $p$ and a subword $v$
of equal length such that there is a $W_I$ automorphism that
takes $v$ or $v^{-1}$ to a word that lexicographically precedes
$p$, then $w$ is not SLPCI$^\pm$ minimal.
We enumerate words of a fixed length by an odometer and check for such
subwords $v$.
If we find such a subword then we increment the odometer
\emph{at the rightmost position of $v$}.
This potentially allows us to skip over large ranges of words that
do not contain any  SLPCI$^\pm$ minimal words. 

As the wordlength grows and exponential growth in the free group
builds up steam, it impractical to hold the
entire SLPCI$^\pm$ graph in memory.
Instead, for each SLPCI$^\pm$ and Whitehead minimal word $w$ we
start constructing its graph component as described in  \fullref{slpci}.
If in this construction we encounter a shortlex predecessor then we
throw $w$ away and proceed to the next candidate.
If no such element occurs then $w$ is minimal in its component.
This procedure would be most effective if the SLPCI$^\pm$ graph
consists of many small components, and if in each component it is easy
to verify whether or not a given word is the shortlex minimal one.
Unfortunately for the latter case, there do exist examples of
components with shortlex local minima. 
For example, here is a component of the graph in rank 2 at length 9
(Capitalization indicates inversion, and the base ordering is $B<A<a<b$.) that contains a word $w:=BBABBAAbA$
that is a shortlex local minimum but not
the global minimum:
\[BBABBAAbA-BBABAbAbA-BBBABBAAA\]
So, to verify that $w$ is not the global minimum in its component we have
to construct the entire component.
That is easy in this example because the component is small.
It turns out that most components are small.
\fullref{fig:histo} shows the observed number of components of each
size in rank 3 at length 15.
In this example $99\%$ of the components have size at most 14.

For all\footnote{The
  formula for the size of the component containing
$C^{L-8}BCACaBAA$ has been confirmed up to $L= 25$, but we have not computed the full component frequency
distribution for $L>15$.} $11\leq L\leq 15$ the component frequency
plot looks much like \fullref{fig:histo}, with most
values clustered left and one prominent spectrum at multiples of
$\frac{1}{2}((L-7)^2+11(L-7)+30)$, with a unique largest component
of size $\frac{L-7}{2}((L-7)^2+11(L-7)+30)$ represented by $C^{L-8}BCACaBAA$.

Myasnikov and Shpilrain \cite{MR2015300} proved that components of the
Whitehead level--$L$ graph in rank $r=2$ have size bounded by a
polynomial of degree $2r-2$
in $L$, see also \cite{MR2077762, MR3364219}, and conjectured that this
should be true in higher ranks (see the
conjecture and discussion following \cite[Corollary~1.2]{MR2015300}).
The conjecture has been proven in some cases with additional technical
hypotheses \cite{MR2230319,MR2266870}.
Myasnikov and Shpilrain also, citing experimental evidence, give a specific quartic polynomial for
rank 3 bounding the size of the largest component, and a
representative of that component. 
Their representative is in the same $\Aut(\F_3)$--orbit as $C^{L-8}BCACaBAA$. 

\begin{figure}[h]
  \centering
  \includegraphics[width=\textwidth]{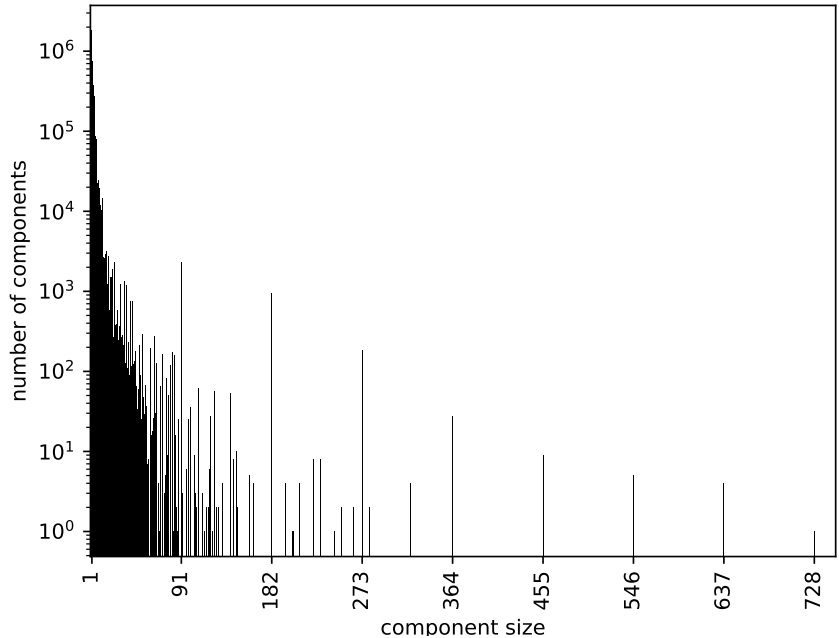}
  \caption{Number of connected components of the SLPCI$^\pm$ graph by size
    in rank 3 at length 15.}
  \label{fig:histo}
\end{figure}

We enumerated $\Aut(\F_r)$ equivalence classes of cyclic subgroup up
to length 16 for $r\leq 4$. 
\fullref{tab:numberofcandidates} shows the resulting number of representatives
of each length.
Lists of these representatives can be found at:

\href{https://www.mat.univie.ac.at/~cashen/orgcensus/}{\nolinkurl{https://www.mat.univie.ac.at/~
  cashen/orgcensus/}}

Our tools for working with free groups and  enumerating equivalence classes are extensions of those
developed with Manning  for \cite{CasMan15}.

\begin{table}[h]
  \centering
  \begin{tabular}[h]{c | c | c | c | c}
      $L$ & $\F_1$&$\F_2$&$\F_3$&$\F_4$\\
           \hline
    1 & 1&0 & 0&0\\
    2 & 1&0 & 0&0\\
    3 & 1&0& 0&0\\
    4 & 1&2&  0&0\\
    5 & 1&3 & 0&0\\
    6 & 1&8 & 1&0\\
    7& 1&12 & 5&0\\
    8& 1&34 & 18&2\\
    9& 1&71 & 98&5\\
    10& 1&217 & 522&35\\
    11& 1&515 & 3,124&315\\
    12& 1&1,423 & 16,866&7,106\\
    13& 1&3,834 & 96,086&93,460\\
    14 &  1&11,816& 582,844&1,124,764\\
    15 & 1&33,321&3,481,458  & 11,679,597\\
    16& 1&95,440&19,514,686  &109,264,221
  \end{tabular}
  \caption{The number of length $L$ $\Aut(\F_r)$ equivalence
    classes of cyclic subgroup not contained in a proper free factor
    of $\F_r$, for $L\leq 16$.}
  \label{tab:numberofcandidates}
\end{table}

\subsection{Computing imprimitivity rank}
We compute imprimitivity rank by inductively building Stallings graphs
$\Gamma$ representing finite rank subgroups $H$ of $\F_r$ containing $w$ as an
imprimitive word.
Since we are interested in minimal rank subgroups containing $w$, we
may assume that the loop labelled by $w$ traverses every edge of
$\Gamma$.
Furthermore, since we are interested in subgroups containing
$w$ as an imprimitive element, we may assume $w$ traverses every edge
at least twice.
In particular, $\Gamma$ can contain at most $\lfloor|w|_a/2\rfloor$
edges labelled $a$ for each basis element $a$.
These constraints cut down on the number of possible graphs $\Gamma$.

\fullref{tab:irank} shows the observed number of equivalence classes
of cyclic subgroup of given
imprimitivity rank at word lengths 14-16.

\begin{table}[h]
  \centering
  \begin{tabular}[h]{c|c | c | c | c | c}
    $L=14$& irank & $\F_1$&$\F_2$&$\F_3$&$\F_4$\\
    \hline
    &1 & 1&12 &5 &0\\
    &2 & 0& 11804&364 &6\\
    &3 & 0&0& 582475&321\\
    &4& 0&0&  0&1124437\\
  \multicolumn{1}{c}{$L=15$}\\
     \hline
    &1 & 1& 3&0 &0\\
    &2 & 0& 33318& 258&7\\
    &3 & 0&0&3481200 &1055\\
       &4 & 0&0&  0&11678535\\
        \multicolumn{1}{c}{$L=16$}\\
           \hline
    &1 & 1&34 & 18&2\\
    &2 & 0&95406 & 2765&111\\
    &3 & 0&0& 19511903&11023\\
    &4 & 0&0&  0&109253085\\
  \end{tabular}
  
  \caption{The number of length $L$ $\Aut(\F_r)$ equivalence
    classes of cyclic subgroup not contained in a proper free factor,
    by rank and imprimitivity rank, for $14\leq L\leq 16$.}
  \label{tab:irank}
\end{table}

Louder and Wilton define  \emph{$w$--subgroups} to be those minimal rank
subgroups containing $w$ as an imprimitive element that are maximal
with respect to inclusion among all such subgroups.
They prove that a word $w$ of imprimitivity rank 2 has a unique $w$--subgroup.
On the other hand, elements of imprimitivity rank $r$ in $\F_r$
obviously have a unique $w$--subgroup, the group $\F_r$ itself.
For intermediate imprimitivity ranks the uniqueness of
$w$--subgroups is an open question.
We observe that all elements in our enumeration have unique $w$--subgroups:
\begin{proposition}\label{uniqueWsubgroups}
  If $w\in \F_4$ has imprimitivity rank 3 and length at most 16 then it
  has a unique $w$--subgroup.
\end{proposition}

Algorithms in this section can be found in
\texttt{imprimitivity\_rank.py} of \href{https://github.com/cashenchris/freegroups.git}{\texttt{github:}cashenchris/freegroups}.

\subsection{Verifying hyperbolicity}
Given an imprimitive, Whitehead minimal word $w\in \F_r$ that is not a
proper power, we check (non)hyperbolicity of $G:=\F_r/\llangle w\rrangle$ using the following
tests:
\begin{enumerate}
\item Check if the presentation is \emph{cyclically pinched}, that is, if it can be
  written as an product of two finite rank free groups
  amalgamated over a cyclic subgroup.
  This is true if a cyclic
  permutation of $w$ can be written as a product $uv$ such that $u$
  and $v$ are nontrivial words with no generators of $\F_r$ in common.
  In this case, $G$ is nonhyperbolic if $u$ and $v$ are both proper
  powers, and hyperbolic otherwise.
If  not cyclically pinched, then\label{item:cyclicallypinched}
\item check if $w$ satisfies the hypotheses of Ivanov and Schupp \cite[Theorem~3 or
  4]{MR1401522}, and if so, whether $G$ is hyperbolic or not. If
  Ivanov-Schupp does not apply, then\label{item:IvanovSchupp}
  \item check if $w$ satisfies one of the small cancellation
    conditions $C(7)$, $C(5)-T(4)$, $C(4)-T(5)$, or $C(3)-T(7)$, in which
    case $G$ is hyperbolic via results of Gersten and Short \cite{GerSho90}. Otherwise,\label{item:smallcancellation}
  \item check if $w$ satisfies the $C'(1/4)-T'$ hyperbolicity condition of
    Blufstein and Minian \cite{BM19}. If not,\label{item:BlufsteinMinian}
   \item check hyperbolicity of $G$ with \texttt{GAP}. Finally, if that
     fails, then\label{item:walrus}
     \item verify hyperbolicity of $G$ with \texttt{kbmag}.\label{item:kbmag}
\end{enumerate}

The algorithm can be found in 
\texttt{geometryofonerelatorgroups.py}
of:

\href{https://github.com/cashenchris/onerelatorgroups.git}{\texttt{github:}cashenchris/onerelatorgroups}

  We remark that the above checks cannot certify a counterexample to the
  Louder-Wilton conjecture, since the only checks that can conclusively return `nonhyperbolic' are 
  \eqref{item:cyclicallypinched} and \eqref{item:IvanovSchupp}.
  It is easy to verify that the nonhyperbolic cyclically pinched case
  implies imprimitivity rank 2, and we checked this for the
  Ivanov-Schupp case in \fullref{prop:ISirank}.
  Thus, the worst that could happen is that we encounter a highly
  imprimitive word whose hyperbolicity we are unable to decide
  with the above tools.
  We did not encounter any such words.
In principle, if $\F_r/\llangle w \rrangle$ hyperbolic, this will be
verified by \texttt{kbmag} \cite{Hol00}, given enough time and
computing resources, but it will run forever in the nonhyperbolic case.
Even in our experiments \texttt{kbmag} took up to several minutes to
succeed, making it unsuitable for checking hundreds
of millions of examples.
Checks \eqref{item:cyclicallypinched}-\eqref{item:walrus} are faster,
but sometimes inconclusive.

Items \eqref{item:cyclicallypinched}-\eqref{item:BlufsteinMinian} we
implemented ourselves.

In item \eqref{item:walrus} we used the function \texttt{IsHyperbolic}
(with parameter $\varepsilon = 1/100$) of the \texttt{GAP} \cite{GAP4}
package \texttt{walrus} \cite{walrus} which is based on an algorithm
of Holt, et al.\
\cite{HolLinNeu19}.
The function tries to verify that the \texttt{RSym} curvature
distribution scheme defined in  \cite{HolLinNeu19}  succeeds on every
van Kampen
diagram over the presentation defined by $w$.
This step is crucial, since although small cancellation words are
generic, there are still far too many words that evade checks
\eqref{item:cyclicallypinched}-\eqref{item:BlufsteinMinian} to
feasibly check with \texttt{kbmag}.
Step \eqref{item:walrus} is based on the second author's investigation of the application of \texttt{RSym} and
its variants to hyperbolicity of one-relator groups \cite{Hof20}. (Another recent application of \texttt{RSym}, to a different class of groups,
was conducted by Chalk \cite{Cha20}.)

The implementation of \texttt{IsHyperbolic} in the version of \texttt{walrus}
we used does not capture the full power of the algorithms described in
\cite{HolLinNeu19}:
\begin{itemize}
\item \texttt{IsHyperbolic} quits and answers inconclusively if it
  encounters certain potential bad van Kampen diagrams, but sometimes
  it can be checked by hand that such a diagram does not really exist.
 \item The \texttt{RSym} algorithm in \cite{HolLinNeu19} takes a \emph{depth}
   parameter $d$. Success for any $d$ implies hyperbolicity.
   \texttt{IsHyperbolic} only implements $d=1$.
 \item \cite{HolLinNeu19}  also defines an enhanced version of
   \texttt{RSym} called $\texttt{RSym}^+$ that is not implemented.
\end{itemize}

The second author showed by hand that the enhanced version of
\texttt{RSym} often succeeds when \texttt{IsHyperbolic} is
inconclusive. For example:
\begin{theorem}[{\cite[Theorem~5.6]{Hof20}}]\label{charlotte}
  If $w\in \F_3$ has imprimitivity rank 3 and length at most 12 then
  $\texttt{RSym}^+$ succeeds at depth 2.
\end{theorem}
We considered implementing an enhanced \texttt{RSym} algorithm, but it
turned out in our experiments that Checks
\eqref{item:cyclicallypinched}-\eqref{item:walrus} caught enough words
that \texttt{kbmag} could finish off the rest in a reasonable amount of time.

\subsection{Word length 17 and beyond}
We have described the experiments up to length 16.
To extend \fullref{rank3} to length 17 we
altered the algorithm.
It turns out that hyperbolicity checks
\eqref{item:cyclicallypinched}-\eqref{item:walrus} are fast compared
to computing equivalence classes and imprimitivity ranks.
Also, the imprimitivity rank computation can be short-circuited to
give a faster decision of whether the imprimitivity rank is greater
than 2.
For length 17 we enumerated SLPCI$^\pm$ and Whitehead minimal words 
and checked
for hyperbolicity using checks
\eqref{item:cyclicallypinched}-\eqref{item:walrus} first.
If some check answered `hyperbolic' we moved on to the next candidate.
Otherwise, we checked if the imprimitivity rank was equal to 2.
If so, we moved on to the next candidate.
In the remaining cases where hyperbolicity was inconclusive and
imprimitivity rank was greater than 2, then we proceeded to check if
the word was the 
the shortlex minimal generator of a cyclic subgroup in its
$\Aut(\F_r)$ equivalence class, and if so verified hyperbolicity with \texttt{kbmag}.

This still took $\sim\! 4$ years of CPU time.
The problem is completely parallelizable over the words of fixed
length in a free group, so conceivably our programs could be run on a
larger cluster to extend the results to length 18 or 19, if
there were any particular reason to expect that a counterexample would
be revealed at these lengths.
We did have a reason to push as far as length 17: in rank 3 at length at most 12, \texttt{kbmag} is not
necessary---checks \eqref{item:cyclicallypinched}-\eqref{item:walrus} always
succeed in verifying hyperbolicity.
We conjectured, and verified, that the same phenomenon would repeat in
rank 4---checks \eqref{item:cyclicallypinched}-\eqref{item:walrus}
suffice up to length 16, but beginning with length 17 additional
complexity appears that requires \texttt{kbmag}.
This leaves us with the question of whether in rank $r$ all highly
imprimitive words of
length at most $4r$ can be verified hyperbolic using only checks
\eqref{item:cyclicallypinched}-\eqref{item:walrus}, or, similarly to
\fullref{charlotte}, using some enhancement of \texttt{RSym}?
If so, this would improve \fullref{ISbound}.

\subsection{Stable commutator length}
The \emph{commutator length (cl)} of an element in the commutator subgroup
of a group is the minimal number of factors in the expression of that
element as a product of commutators.
The \emph{stable commutator length (scl)} is $\text{scl}(w):=\lim_{n\to\infty}\text{cl}(w^n)/n$.
Heuer \cite[Conjecture~6.3.2]{Heu19} conjectures a generalization of
the Duncan-Howie scl-gap theorem \cite{MR1128707}  saying that
$\text{scl}\geq (\text{irank}-1)/2$.
We confirm Heuer's conjecture on our dataset:

\begin{proposition}\label{scl}
  For all nontrivial $w$ in the commutator subgroup of $\F_4$ with $|w|\leq 16$,
  we have $\text{scl}(w)\geq (\text{irank}(w)-1)/2$.
\end{proposition}
We computed stable commutator lengths with \texttt{scallop}
\cite{CalWal14}. The results are shown in \fullref{fig:scl}.

\begin{figure}[h]
  \centering
  \includegraphics[width=\textwidth]{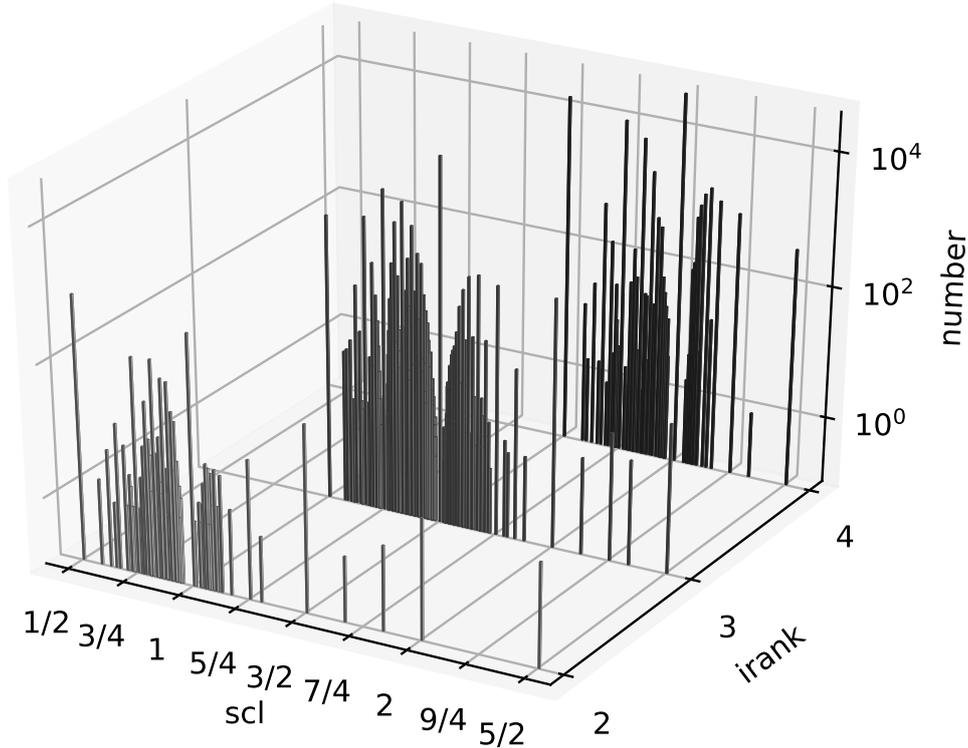}
  \caption{The number of $\Aut(\F_4)$ equivalence classes of cyclic
    subgroups of length at most 16 by stable commutator length and
    imprimitivity rank.}
  \label{fig:scl}
\end{figure}

\bibliographystyle{hypersshort}
\bibliography{orghyp}

\providecommand{\bysame}{\leavevmode\hbox to3em{\hrulefill}\thinspace}
\providecommand{\MR}{\relax\ifhmode\unskip\space\fi MR }
\providecommand{\MRhref}[2]{%
  \href{http://www.ams.org/mathscinet-getitem?mr=#1}{#2}
}
\providecommand{\href}[2]{#2}
\providecommand{\doi}[1]{doi: #1}
\begin{thebibliography}{10}

\bibitem{BM19}
M.~A. Blufstein and E.~G. Minian,
  \emph{\href{https://arxiv.org/pdf/1907.06738.pdf}{Strictly systolic angled
  complexes and hyperbolicity of one-relator groups}}, preprint (2019),
  \href{http://arXiv.org/abs/1907.06738}{{\texttt{arXiv:1907.06738}}}.

\bibitem{CalWal14}
D.~Calegari and A.~Walker,
  \emph{\href{https://github.com/aldenwalker/scallop}{{\texttt{scallop}}}},
  computer program (2014),
  \href{https://github.com/aldenwalker/scallop}{{\texttt{\detokenize{https://github.com/aldenwalker/scallop}}}}.

\bibitem{CasMan15}
C.~H. Cashen and J.~F. Manning,
  \emph{\href{http://dx.doi.org/10.1112/S1461157015000108}{Virtual geometricity
  is rare}}, LMS J. Comput. Math. \textbf{18} (2015), no.~1, 444--455.

\bibitem{Cha20}
C.~Chalk, \emph{\href{https://arxiv.org/pdf/2005.10653.pdf}{Fibonacci groups,
  {$F(2,n)$}, are hyperbolic for {$n$} odd and {$n >= 11$}}}, preprint (2020),
  \href{http://arXiv.org/abs/2005.10653}{{\texttt{arXiv:2005.10653}}}.

\bibitem{MR3364219}
B.~Cooper and E.~Rowland,
  \emph{\href{http://dx.doi.org/10.1090/conm/633/12648}{Classification of
  automorphic conjugacy classes in the free group on two generators}},
  Algorithmic problems of group theory, their complexity, and applications to
  cryptography, Contemp. Math., vol. 633, Amer. Math. Soc., Providence, RI,
  2015, pp.~13--40.

\bibitem{MR1128707}
A.~J. Duncan and J.~Howie,
  \emph{\href{http://dx.doi.org/10.1007/BF02571522}{The genus problem for
  one-relator products of locally indicable groups}}, Math. Z. \textbf{208}
  (1991), no.~2, 225--237.

\bibitem{GAP4}
The GAP~Group, \emph{\href{https://www.gap-system.org}{{GAP} -- {G}roups,
  {A}lgorithms, and {P}rogramming, version 4.10.2}}, 2019,
  \href{https://www.gap-system.org}{{\texttt{\detokenize{https://www.gap-system.org}}}}.

\bibitem{MR3876736}
G.~Gardam and D.~J. Woodhouse,
  \emph{\href{http://dx.doi.org/10.1090/proc/14238}{The geometry of one-relator
  groups satisfying a polynomial isoperimetric inequality}}, Proc. Amer. Math.
  Soc. \textbf{147} (2019), no.~1, 125--129.

\bibitem{GerSho90}
S.~M. Gersten and H.~B. Short,
  \emph{\href{http://dx.doi.org/10.1007/BF01233430}{Small cancellation theory
  and automatic groups}}, Invent. Math. \textbf{102} (1990), no.~2, 305--334.

\bibitem{MR3893306}
N.~Gupta and I.~Kapovich,
  \emph{\href{http://dx.doi.org/10.1017/S0305004117000755}{The primitivity
  index function for a free group, and untangling closed curves on hyperbolic
  surfaces}}, Math. Proc. Cambridge Philos. Soc. \textbf{166} (2019), no.~1,
  83--121, With an appendix by Khalid Bou-Rabee.

\bibitem{Heu19}
N.~Heuer, \emph{Constructions in stable commutator length and bounded
  cohomology}, Ph.D. thesis, University of Oxford, 2019.

\bibitem{Hof20}
C.~Hoffmann, \emph{Generalisations of small cancellation: The {RSym} algorithm
  on hyperbolic one-relator groups}, Master's thesis, University of Vienna,
  2020.

\bibitem{Hol00}
D.~Holt,
  \emph{\href{http://homepages.warwick.ac.uk/~mareg/download/kbmag2/}{kbmag --
  {K}nuth-{B}endix on monoids and automatic groups}}, 2000,
  \href{http://homepages.warwick.ac.uk/~mareg/download/kbmag2/}{{\texttt{\detokenize{http://homepages.warwick.ac.uk/~mareg/download/kbmag2/}}}}.

\bibitem{HolLinNeu19}
D.~Holt, S.~Linton, M.~Neunhoeffer, R.~Parker, M.~Pfeiffer, and C.~M.
  Roney-Dougal,
  \emph{\href{https://arxiv.org/pdf/1905.09770.pdf}{Polynomial-time proofs that
  groups are hyperbolic}}, preprint (2019),
  \href{http://arXiv.org/abs/1905.09770}{{\texttt{arXiv:1905.09770}}}.

\bibitem{MR1401522}
S.~V. Ivanov and P.~E. Schupp,
  \emph{\href{http://dx.doi.org/10.1090/S0002-9947-98-01818-2}{On the
  hyperbolicity of small cancellation groups and one-relator groups}}, Trans.
  Amer. Math. Soc. \textbf{350} (1998), no.~5, 1851--1894.

\bibitem{MR1882114}
I.~Kapovich and A.~Myasnikov,
  \emph{\href{http://dx.doi.org/10.1006/jabr.2001.9033}{Stallings foldings and
  subgroups of free groups}}, J. Algebra \textbf{248} (2002), no.~2, 608--668.

\bibitem{MR2107437}
I.~Kapovich and P.~Schupp,
  \emph{\href{http://dx.doi.org/10.1007/s00208-004-0570-x}{Genericity, the
  {A}rzhantseva-{O}l\cprime shanskii method and the isomorphism problem for
  one-relator groups}}, Math. Ann. \textbf{331} (2005), no.~1, 1--19.

\bibitem{MR2221020}
I.~Kapovich, P.~Schupp, and V.~Shpilrain,
  \emph{\href{http://dx.doi.org/10.2140/pjm.2006.223.113}{Generic properties of
  {W}hitehead's algorithm and isomorphism rigidity of random one-relator
  groups}}, Pacific J. Math. \textbf{223} (2006), no.~1, 113--140.

\bibitem{MR2077762}
B.~Khan, \emph{\href{http://dx.doi.org/10.1090/conm/349/06360}{The structure of
  automorphic conjugacy in the free group of rank two}}, Computational and
  experimental group theory, Contemp. Math., vol. 349, Amer. Math. Soc.,
  Providence, RI, 2004, pp.~115--196.

\bibitem{MR2230319}
D.~Lee, \emph{\href{http://dx.doi.org/10.1016/j.jalgebra.2006.04.012}{Counting
  words of minimum length in an automorphic orbit}}, J. Algebra \textbf{301}
  (2006), no.~1, 35--58.

\bibitem{MR2266870}
D.~Lee, \emph{\href{http://dx.doi.org/10.1016/j.jalgebra.2006.03.038}{A tighter
  bound for the number of words of minimum length in an automorphic orbit}}, J.
  Algebra \textbf{305} (2006), no.~2, 1093--1101.

\bibitem{LouWil18a}
L.~Louder and H.~Wilton, \emph{\href{https://arxiv.org/pdf/1803.02671}{Negative
  immersions for one-relator groups}}, preprint (2018),
  \href{http://arXiv.org/abs/1803.02671}{{\texttt{arXiv:1803.02671}}}.

\bibitem{MR0399269}
J.~McCool and A.~Pietrowski, \emph{On a conjecture of {W}. {M}agnus}, Word
  problems: decision problems and the {B}urnside problem in group theory
  ({C}onf., {U}niv. {C}alifornia, {I}rvine, {C}alif. 1969; dedicated to {H}anna
  {N}eumann), 1973, pp.~453--456. Studies in Logic and the Foundations of
  Math., Vol. 71.

\bibitem{MR2015300}
A.~G. Myasnikov and V.~Shpilrain,
  \emph{\href{http://dx.doi.org/10.1016/S0021-8693(03)00339-9}{Automorphic
  orbits in free groups}}, J. Algebra \textbf{269} (2003), no.~1, 18--27.

\bibitem{MR0222152}
B.~B. Newman,
  \emph{\href{http://dx.doi.org/10.1090/S0002-9904-1968-12012-9}{Some results
  on one-relator groups}}, Bull. Amer. Math. Soc. \textbf{74} (1968), 568--571.

\bibitem{ipython}
F.~P\'erez and B.~E. Granger,
  \emph{\href{http://dx.doi.org/10.1109/MCSE.2007.53}{{IP}ython: a system for
  interactive scientific computing}}, Computing in Science and Engineering
  \textbf{9} (2007), no.~3, 21--29.

\bibitem{walrus}
M.~Pfeiffer,
  \emph{\href{https://www.gap-system.org/Packages/walrus.html}{walrus - a {GAP}
  package, version 0.99}}, 2019,
  \href{https://www.gap-system.org/Packages/walrus.html}{{\texttt{\detokenize{https://www.gap-system.org/Packages/walrus.html}}}}.

\bibitem{MR3265279}
D.~Puder, \emph{\href{http://dx.doi.org/10.1007/s11856-013-0055-2}{Primitive
  words, free factors and measure preservation}}, Israel J. Math. \textbf{201}
  (2014), no.~1, 25--73.

\bibitem{Whi36}
J.~H.~C. Whitehead, \emph{\href{http://www.jstor.org/stable/1968618}{On
  equivalent sets of elements in a free group}}, Ann. of Math. (2) \textbf{37}
  (1936), no.~4, 782--800.

\end{thebibliography}

\end{document}